\documentclass[11pt,reqno]{amsart}
\usepackage[latin1]{inputenc} 
\usepackage{indentfirst}
\usepackage{epsfig} 
\usepackage{graphicx} 
\usepackage{multirow} 
\usepackage{verbatim} 
\usepackage{rotating} 
\usepackage{lettrine}
\usepackage{lscape}
\usepackage{ae}
\usepackage{float}
\usepackage[all]{xy}
\usepackage{icomma} 
\usepackage{varioref} 
\usepackage{layout}
\usepackage[Lenny]{fncychap} 
\usepackage{enumerate}
\usepackage{amsfonts, amsmath, amssymb, stmaryrd, latexsym, amsthm, }
\usepackage{dsfont}
\usepackage{wrapfig, color, epsfig}
\usepackage{xspace, setspace}
\usepackage{array}
\usepackage{longtable}
\usepackage{geometry}
\usepackage[frenchb,english]{babel}
\begin{document}
\newtheorem{The}{Theorem}[section]

\newcommand{\K}{\mathbb{Q}(\sqrt{d},i)}
\newcommand{\q}{\mathbb{Q}(\sqrt{d})/\mathbb{Q}}
\newcommand{\qq}{\mathbb{Q}(\sqrt{d})}
\newcommand{\pro}{\prod_{p/\Delta_K}e(p)}
\newtheorem{lem}[The]{Lemma}
\newtheorem{propo}[The]{Proposition}
\newtheorem{coro}[The]{Corollary}
\newtheorem{proprs}[The]{properties}
\newtheorem*{mainthem}{Theorem}
\theoremstyle{remark}
\newtheorem{rema}[The]{\bf Remark}
\newtheorem{remas}[The]{\bf Remarks}
\newtheorem{exam}[The]{\textbf{Numerical Example}}
\newtheorem{exams}[The]{\textbf{Numerical Examples}}
\newtheorem{df}[The]{definition}
\newtheorem{dfs}[The]{definitions}
\def\NN{\mathds{N}}
\def\RR{\mathbb{R}}
\def\HH{I\!\! H}
\def\QQ{\mathbb{Q}}
\def\CC{\mathbb{C}}
\def\ZZ{\mathbb{Z}}
\def\OO{\mathcal{O}}
\def\kk{\mathds{k}}
\def\KK{\mathbb{K}}
\def\FF{\mathbb{F}}
\def\ho{\mathcal{H}_0^{\frac{h(d)}{2}}}
\def\LL{\mathbb{L}}
\def\o{\mathcal{H}_0}
\def\h{\mathcal{H}_1}
\def\hh{\mathcal{H}_2}
\def\hhh{\mathcal{H}_3}
\def\hhhh{\mathcal{H}_4}
\def\k{\mathds{k}^{(*)}}
\def\G{\mathds{k}^{(*)}}
\def\l{\mathds{L}}
\def\L{\kk_2^{(2)}}
\def\kkk{k^{(*)}}
\def\Q{\mathcal{Q}}
\def\D{\kk_2^{(1)}}
\title[Capitulation in the absolutely abelian  extensions...]{Capitulation in the absolutely abelian  extensions of some fields $\QQ(\sqrt{p_1p_2q}, \sqrt{-1})$}
\author[A. Azizi]{Abdelmalek Azizi}
\address{Abdelmalek Azizi: Mohammed First University, Mathematics Department, Sciences Faculty, Oujda, Morocco }
\email{abdelmalekazizi@yahoo.fr}
\author[A. Zekhnini]{Abdelkader Zekhnini}
\address{Abdelkader Zekhnini: Mohammed First University, Mathematics Department, Pluridisciplinary faculty, Nador, Morocco}
\email{zekha1@yahoo.fr}
\author[M. Taous]{Mohammed Taous}
\address{Mohammed Taous: Moulay Ismail University, Mathematics Department, Sciences and Techniques Faculty, Errachidia, Morocco.}
\email{taousm@hotmail.com}

\maketitle
\selectlanguage{english}
\begin{abstract}
We study the capitulation of $2$-ideal classes of an infinite
family of imaginary bicyclic biquadratic number fields consisting of
 fields $\kk =\QQ(\sqrt{p_1p_2q}, i)$, where $i=\sqrt{-1}$ and $p_1\equiv p_2\equiv-q\equiv1 \pmod 4$
are different primes. For each of the three quadratic extensions $\KK/\kk$
inside the absolute genus field $\kk^{(*)}$ of $\kk$, we compute the capitulation
kernel of $\KK/\kk$. Then we deduce that each strongly ambiguous class of $\kk/\QQ(i)$
capitulates already in $\kk^{(*)}$, which is  smaller than the relative
genus field $(\kk/\QQ(i))^*$.
\end{abstract}
\section{Introduction and Notations}\label{sec1}
Let $k$ be an algebraic number field and let $\mathbf{C}l_2(k)$ denote its 2-class group, that is the 2-Sylow
subgroup of the ideal class group, $\mathbf{C}l(k)$,  of $k$.
We denote by  $k^{(*)}$  the absolute genus field of $k$.
Suppose $F$ is a finite extension of $k$,
then we say that an ideal class of $k$ capitulates in $F$ if it is in the kernel of the
homomorphism
$$J_F: \mathbf{C}l(k) \longrightarrow \mathbf{C}l(F)$$ induced by extension of ideals from $k$
to $F$. An
important problem in Number Theory is to determine explicitly the kernel of  $J_F$, which is
usually called the capitulation kernel. If $F$ is
the  relative genus  field of a cyclic extension $K/k$, which we denote  by $(K/k)^*$ and that is  the maximal unramified extension  of  $K$ which is obtained by composing  $K$ and an abelian extension over $k$, F. Terada  states in  \cite{FT-71} that all the ambiguous ideal classes of $K/k$, which are classes of $K$ fixed under any element of $\mathrm{G}al(K/k)$,  capitulate in
 $(K/k)^*$. If $F$ is the  absolute genus  field of an abelian extension $K/\QQ$, then H. Furuya confirms in \cite{Fu-77} that every   strongly ambiguous  class of $K/\QQ$, that is an ambiguous ideal class containing at least one ideal invariant under any element of $\mathrm{G}al(K/\QQ)$, capitulates in  $F$. In this paper, we construct a family of number fields $k$ for which  all the strongly ambiguous  classes of $k/\QQ(i)$ capitulate in  $k^{(*)}\varsubsetneq (k/\QQ(i))^*$.\par
 Let $\kk=\QQ(\sqrt{d}, i)$ and $\mathds{K}$ be an unramified quadratic extension of $\kk$ that is abelian  over $\QQ$. Denote by $\mathrm{A}m_s(\kk/\QQ(i))$
 the group of the strongly ambiguous classes of  $\kk/\QQ(i)$. In \cite{AZT14-1}, we studied the capitulation problem in the absolutely abelian extensions of $\kk$ for  $d=2pq$ and $p\equiv q\equiv1\pmod4$ are different primes, and in  \cite{AZT14-3}, we have dealt with  the same problem assuming $p\equiv -q\equiv1\pmod4$. In  \cite{AZT14-2, AZTM, AZT-15} and under the assumption $\mathbf{C}l_2(\kk)\simeq (2, 2, 2)$, we studied the capitulation problem of the $2$-ideal classes of $\kk$ in its fourteen unramified extensions, within the first Hilbert $2$-class field of $\kk$, and we gave the abelian type invariants of the $2$-class groups of these fourteen fields, additionally we determined  the structure of the metabelian Galois group $G=\mathrm{Gal}(\L/\kk)$ of the second Hilbert 2-class field $\L$ of $\kk$.   Let $p_1\equiv p_2\equiv-q\equiv1 \pmod 4$ be different primes and  $d=p_1p_2q$, it is the purpose of the present article
to pursue this research project. We will  compute  the capitulation
kernel of $\KK/\kk$ and  we will  deduce  that   $\mathrm{A}m_s(\kk/\QQ(i))\subseteq \ker J_{\k}$. As
an application we will determine these kernels when  $\mathbf{C}l_2(\mathds{k})$ is of type $(2, 2, 2)$.

Let  $k$ be a number field, during this paper, we adopt the following notations:
 \begin{itemize}
   \item   $p_1\equiv p_2\equiv-q\equiv1 \pmod 4$ are different primes.
   \item $\kk$: denotes the field $\QQ(\sqrt{p_1p_2q},\sqrt{-1})$.
   \item $\kappa_{K}$: the capitulation kernel of an unramified extension $K/\kk$.
   \item $\mathcal{O}_k$: the ring of integers of $k$.
   \item $E_k$: the unit group of $\mathcal{O}_k$.
   \item  $W_k$: the group of roots of unity contained in $k$.
   \item $\mathrm{F}.\mathrm{S}.\mathrm{U}$ : the fundamental system of units.
   \item $k^+$: the maximal real subfield of $k$, if it is a CM-field.
   \item $Q_k=[E_k:W_kE_{k^+}]$ is Hasse's unit index, if $k$ is a CM-field.
   \item $q(k/\QQ)=[E_k:\prod_{i=1}^{s} E_{k_i}]$ is the unit index of $k$, if $k$ is multiquadratic, where $k_1$, ..., $k_s$  are the  quadratic subfields of $k$.
  \item $k^{(*)}$: the absolute genus field of $k$.
  \item $\mathbf{C}l_2(k)$: the 2-class group of $k$.
  \item $i=\sqrt{-1}$.
  \item $\epsilon_m$: the fundamental unit of $\QQ(\sqrt m)$, if $m>1$ is a square-free integer.
  \item $N(a)$: denotes the absolute norm of a number $a$ i.e. $N_{k/\QQ}(a)$, where $k=\QQ(\sqrt a)$.
  \item $x\pm y$ means $x+y$ or $x-y$ for some numbers $x$ and $y$.
 \end{itemize}
  Our main theorem is.
 \begin{mainthem}
 Let $p_1\equiv p_2\equiv-q\equiv1 \pmod 4$ be different primes and put $\kk=\QQ(\sqrt{p_1p_2q}, \sqrt{-1})$.
Denote by $Am_s(\kk/\QQ(i))$
the group of the strongly ambiguous classes of  $\kk/\QQ(i)$.
If $\mathds{K}$ is an unramified quadratic extension of $\kk$ that is  abelian  over $\QQ$,
then
\begin{enumerate}[\upshape\indent1.]
\item $\left|\kappa_{\KK}\right|=2, 4\text{ or }8$.
  \item $\kappa_{\KK}\subset Am_s(\kk/\QQ(i))$.
  \item $Am_s(\kk/\QQ(i))\subseteq \kappa_{\k}$.
\end{enumerate}
\end{mainthem}
\section{\bf{Preliminary results}}
Let us  first collect some results that will be useful in what follows.\par
Let $k_j$, $1\leq j\leq 3$, be the three real quadratic subfields of a biquadratic bicyclic real number field $K_0$
and $\epsilon_j>1$ be the fundamental unit of $k_j$. Since $\alpha^2N_{K_0/{\bf Q}}(\alpha) =\prod_{j=1}^3N_{K_0/k_j}(\alpha)$ for any $\alpha\in K_0$,
the square of any unit of $K_0$ is in the group generated by the $\epsilon_j$'s, $ 1\leq j\leq 3$.
Hence, to determine a fundamental system of units of $K_0$ it suffices to determine which of the units in
$B:=\{\epsilon_1,\epsilon_2,\epsilon_3,\epsilon_1\epsilon_2,\epsilon_1\epsilon_3,\epsilon_2\epsilon_3,\epsilon_1\epsilon_2\epsilon_3\}$
are  squares in $K_0$ (see \cite{Wa-66} or \cite{Lo-00}).
Put  $K=K_0(i)$,  then  to determine a $\mathrm{F}.\mathrm{S}.\mathrm{U}$  of $K$,  we will use the following result (see \cite[p.18]{Az-99})  that the first author has deduced from a theorem of Hasse \cite[\S 21, Satz 15 ]{Ha-52}.
\begin{lem}\label{3}
Let $n\geq2$ be an integer   and $\xi_n$ a $2^n$-th primitive root of unity, then
$$\begin{array}{lllr}\xi_n = \dfrac{1}{2}(\mu_n + \lambda_ni), &\hbox{ where }&\mu_n =\sqrt{2 +\mu_{n-1}},&\lambda_n =\sqrt{2 -\mu_{n-1}}, \\
                               &    & \mu_2=0, \lambda_2=2 &\hbox{ and\quad } \mu_3=\lambda_3=\sqrt 2.
                                              \end{array}
                                            $$
Let $n_0$ be the greatest integer such that  $\xi_{n_0}$ is contained in $K$, $\{\epsilon'_1, \epsilon'_2, \epsilon'_3\}$ a $\mathrm{F}.\mathrm{S}.\mathrm{U}$  of $K_0$ and $\epsilon$ a unit of  $K_0$ such that $(2 + \mu_{n_0})\epsilon$ is a square in  $K_0$ $($if it exists$)$. Then a $\mathrm{F}.\mathrm{S}.\mathrm{U}$  of  $K$ is one of the following systems:
\begin{enumerate}[\rm1.]
\item $\{\epsilon'_1, \epsilon'_2, \epsilon'_3\}$ if $\epsilon$ does not exist,
\item $\{\epsilon'_1, \epsilon'_2, \sqrt{\xi_{n_0}\epsilon}\}$ if $\epsilon$
exists;  in this case $\epsilon = {\epsilon'_1}^{i_1}{
\epsilon'_2}^{i_2}\epsilon'_3$, where $i_1$, $i_2\in \{0, 1\}$ (up to a permutation).
\end{enumerate}
\end{lem}
\begin{lem}[\cite{Az-00}, {Lemma 5}]\label{5}
Let $d>1$ be a square-free integer and $\epsilon_d=x+y\sqrt d$,
where $x$, $y$ are  integers or semi-integers. If $N(\epsilon_d)=1$, then $2(x+1)$, $2(x-1)$, $2d(x+1)$ and
$2d(x-1)$ are not squares in  $\QQ$.
\end{lem}
\begin{lem}[\cite{Az-00}, {Lemma 6}]\label{1:048}
Let $q\equiv-1\pmod4$ be a prime and $\epsilon_q=x+y\sqrt q$ be the fundamental unit of  $\QQ(\sqrt q)$. Then $x$ is an even integer, $x\pm1$ is a square in  $\NN$ and $2\epsilon_q$ is a square in $\QQ(\sqrt q)$.
\end{lem}
\begin{lem}[\cite{Az-00}, {Lemma 7}] \label{4}
Let  $p$ be an odd  prime and $\epsilon_{2p}=x+y\sqrt{2p}$.
If  $N(\epsilon_{2p})=1$, then $x\pm1$ is a square in $\NN$ and $2\epsilon_{2p}$
is a square in $\QQ(\sqrt{2p})$.
\end{lem}
\begin{lem}[\cite{Az-99}, {3.(1) p.19}]\label{6}
Let $d>2$ be a square-free integer and $k=\QQ(\sqrt d,i)$, put $\epsilon_d=x+y\sqrt d$.
\begin{enumerate}[\rm\indent1.]
  \item If $N(\epsilon_d)=-1$, then $\{\epsilon_d\}$ is a $\mathrm{F}.\mathrm{S}.\mathrm{U}$  of $k$.
  \item If $N(\epsilon_d)=1$, then $\{\sqrt{i\epsilon_d}\}$ is a $\mathrm{F}.\mathrm{S}.\mathrm{U}$  of $k$ if
   and only if $x\pm1$ is a square in $\NN$ i.e. $2\epsilon_d$ is a square in $\QQ(\sqrt d)$. Else $\{\epsilon_d\}$ is a $\mathrm{F}.\mathrm{S}.\mathrm{U}$  of $k$ $($this result is also in \cite{Kub-56}$)$.
\end{enumerate}
\end{lem}
 \begin{lem}\label{3:105}
 Let $d\equiv1\pmod4$ be a positive square free integer and   $\varepsilon_d=x+y\sqrt d$ be the fundamental unit of  $\QQ(\sqrt d)$. Assume   $N(\varepsilon_d)=1$, then
 \begin{enumerate}[\rm\indent1.]
   \item $x+1$ and $x-1$ are not squares in  $\NN$ i.e. $2\varepsilon_{d}$ is not a square in  $\QQ(\sqrt{d})$.
   \item For all prime  $p$ dividing   $d$, $p(x+1)$ and $p(x-1)$ are not squares in  $\NN$.
 \end{enumerate}
 \end{lem}
 \begin{proof}
 $1.$  As $d\equiv 1 \pmod 4$, then by  \cite[Corollaire 3.2]{AzTa-08} the unit index of  $\QQ(\sqrt{d},i)$ is equal to $1$,  hence  by \cite[Applications (ii)]{Az-99}
  we get that  $2\varepsilon_{d}$ is not a square in  $\QQ(\sqrt{d})$, this is equivalent to  $x+1$ and $x-1$ are not squares in  $\NN$.\\
 \indent $2.$ Assume $p(x+1)$ or $p(x-1)$ is a square in  $\NN$, then, by the decomposition uniqueness in $\ZZ$, there exist $y_1$, $y_2$ in $\ZZ$ such that
   $$\left\{\begin{array}{ll}
 x\pm1=py_1^2,\\
 x\mp1=d'y_2^2;
 \end{array}\right.
 \text{ and  }\left\{\begin{array}{ll}
 y=y_1y_2,\\
 d=pd';
 \end{array}\right.$$
  thus $p(x\pm1)=p^2y_1^2$ and  $p(x\mp1)=p^2y_1^2\mp2p$. This in turn yields that $p^2(x^2-1)=p^2y_1^2(p^2y_1^2\mp2p)$; as  $x^2-1=y^2d$, so we get $y^2d=y_1^2(p^2y_1^2\mp2p)$, and $y_2^2d=p^2y_1^2\mp2p$. Since $d\equiv1\pmod4$ and $p\equiv\pm1\pmod4$, we deduce that $\mp2\equiv y_1^2-y_2^2 \pmod4$. On the other hand, we know  that $a^2\equiv 0$ or $1 \pmod4$ for all $a\in\ZZ$, thus $\mp2\equiv0$,  $1$ or $-1\pmod4$. Which is absurd.
 \end{proof}
\section{\textbf{$\mathrm{F}.\mathrm{S}.\mathrm{U}$  OF SOME CM-FIELDS}}
As  $\kk=\QQ(\sqrt{p_1p_2q}, i)$, so $\kk$ admits three unramified quadratic extensions that are abelian over $\QQ$, which are $\KK_1=\kk(\sqrt{p_1})=\QQ(\sqrt{p_1}, \sqrt{p_2q}, i)$,   $\KK_2=\kk(\sqrt{p_2})=\QQ(\sqrt{p_2}, \sqrt{p_1q}, i)$ and $\KK_3=\kk(\sqrt{q})=\QQ(\sqrt{q}, \sqrt{p_1p_2}, i)$. Put $\epsilon_{p_1p_2q}=x+y\sqrt{p_1p_2q}$. In what follows, we determine the $\mathrm{F}.\mathrm{S}.\mathrm{U}$'s of $\KK_j$, $1\leq j\leq3$.
\subsection{\textbf{$\mathrm{F}.\mathrm{S}.\mathrm{U}$  of the field  $\KK_1$}}
Let $\KK_1=\kk(\sqrt{p_1})=\QQ(\sqrt{p_1}, \sqrt{p_2q}, i)$.
\begin{propo}\label{27}
Keep the previous notations and  put $\epsilon_{p_2q}=a+b\sqrt {p_2q}$.
\begin{enumerate}[\upshape\indent1.]
\item If $a\pm1$ and $(x\pm1$ or $p_1(x\pm1))$  are squares in  $\NN$,  then $\left\{\epsilon_{p_1},  \epsilon_{p_2q}, \sqrt{\epsilon_{p_2q}\epsilon_{p_1p_2q}}\right\}$ is a $\mathrm{F}.\mathrm{S}.\mathrm{U}$ of  $\KK_1^+$ and that of  $\KK_1$ is    $\left\{\epsilon_{p_1}, \sqrt{i\epsilon_{p_2q}}, \sqrt{\epsilon_{p_2q}\epsilon_{p_1p_2q}}\right\}$; moreover,  $Q_{\KK_1}=2$.
\item If $x\pm1$ or $p_1(x\pm1)$  is a square in $\NN$ and  $a\pm1$ is not, then $\left\{\epsilon_{p_1},  \epsilon_{p_2q}, \epsilon_{p_1p_2q}\right\}$ is a $\mathrm{F}.\mathrm{S}.\mathrm{U}$ $\KK_1^+$ and that of   $\KK_1$ is  $\left\{\epsilon_{p_1},  \epsilon_{p_2q}, \sqrt{i\epsilon_{p_1p_2q}}\right\}$; moreover,  $Q_{\KK_1}=2$.
\item If $a\pm1$ and $2p_1(x\pm1)$  are squares in $\NN$,  then $\left\{\epsilon_{p_1},  \epsilon_{p_2q}, \sqrt{\epsilon_{p_1p_2q}}\right\}$ is a $\mathrm{F}.\mathrm{S}.\mathrm{U}$ of $\KK_1^+$ and that of   $\KK_1$ is  $\left\{\epsilon_{p_1}, \sqrt{i\epsilon_{p_2q}}, \sqrt{\epsilon_{p_1p_2q}}\right\}$; moreover,  $Q_{\KK_1}=2$.
\item If $2p_1(x\pm1)$  is a square in $\NN$ and $a\pm1$ is not, then $\left\{\epsilon_{p_1},  \epsilon_{p_2q}, \sqrt{\epsilon_{p_1p_2q}}\right\}$ is a $\mathrm{F}.\mathrm{S}.\mathrm{U}$ of both of $\KK_1^+$ and  $\KK_1$; moreover,  $Q_{\KK_1}=1$.
\item $\left\{\epsilon_{p_1},  \epsilon_{p_2q}, \epsilon_{p_1p_2q}\right\}$ is a  $\mathrm{F}.\mathrm{S}.\mathrm{U}$ of $\KK_1^+$,  $\left\{\epsilon_{p_1},  \sqrt{i\epsilon_{p_2q}}, \epsilon_{p_1p_2q}\right\}$ is a  $\mathrm{F}.\mathrm{S}.\mathrm{U}$ of $\KK_1$ and   $Q_{\KK_1}=2$  if one of the following assertions is satisfied:
\begin{enumerate}[\upshape\indent i.]
\item  $a\pm1$ and $(p_2(x\pm1)$ or $2p_2(x\pm1))$  are squares in $\NN$.
\item  $a\pm1$ and $(q(x\pm1)$ or $2q(x\pm1))$  are squares in $\NN$.
\end{enumerate}
 \item $\left\{\epsilon_{p_1},  \epsilon_{p_2q}, \sqrt{\epsilon_{p_2q}\epsilon_{p_1p_2q}}\right\}$ is a  $\mathrm{F}.\mathrm{S}.\mathrm{U}$ of both of $\KK_1^+$ and $\KK_1$, and     $Q_{\KK_1}=1$  if one of the following assertions is satisfied:
\begin{enumerate}[\upshape\indent i.]
\item  $p_2(a\pm1)$ and $(p_2(x\pm1)$ or $q(x\pm1))$  are squares in $\NN$.
\item  $2p_2(a\pm1)$ and $(2p_2(x\pm1)$ or $2q(x\pm1))$  are squares in $\NN$.
\end{enumerate}
 \item $\left\{\epsilon_{p_1},  \epsilon_{p_2q}, \epsilon_{p_1p_2q}\right\}$ is a  $\mathrm{F}.\mathrm{S}.\mathrm{U}$ of $\KK_1^+$,  $\left\{\epsilon_{p_1},  \epsilon_{p_2q}, \sqrt{i\epsilon_{p_2q}\epsilon_{p_1p_2q}}\right\}$ is a  $\mathrm{F}.\mathrm{S}.\mathrm{U}$  of $\KK_1$ and $Q_{\KK_1}=2$  if one of the following assertions is satisfied:
\begin{enumerate}[\upshape\indent i.]
\item  $p_2(a\pm1)$ and $(2p_2(x\pm1)$ or $2q(x\pm1))$  are squares in $\NN$.
\item  $2p_2(a\pm1)$ and $(p_2(x\pm1)$ or $q(x\pm1))$  are squares in  $\NN$.
\end{enumerate}
\end{enumerate}
\end{propo}
\begin{proof}
 As $N(\epsilon_{p_1})=-1$, so  only $\epsilon_{p_2q}$, $\epsilon_{p_1p_2q}$ and $\epsilon_{p_2q}\epsilon_{p_1p_2q}$ can be squares in  $\KK_1^+$.

 1. Let $\epsilon_{p_2q}=a+b\sqrt {p_2q}$, where $a$ and $b$  are integers of different parities satisfying    $a^2-1=p_2qb^2$, hence  $(a\pm1)(a\mp1)=p_2qb^2$ and gcd of $a\pm1$,  $a\mp1$ divides  2. Thus by the decomposition uniqueness in $\ZZ$ and by Lemma \ref{5}, there exist $b_1$,  $b_2$ in $\ZZ$ such that:
 \begin{enumerate}[\rm\indent a.]
   \item If $a\pm1$ is a square in $\NN$, then
$\left\{\begin{array}{rl}
    a\pm1 &=b_1^2,\\
    a\mp1 &=b_2^2p_2q,
    \end{array}\right.$
    so $\sqrt{\epsilon_{p_2q}}=\frac{1}{\sqrt2}(b_1+b_2\sqrt{p_2q})$, and thus  $\epsilon_{p_2q}$ is not a square in $\KK_1$ but $2\epsilon_{p_2q}$ is.
   \item If $p_2(a\pm1)$ is a square in $\NN$, then
$\left\{\begin{array}{rl}
    a\pm1 &=b_1^2p_2\\
    a\mp1 &=b_2^2q,
    \end{array}\right.$
  so $\sqrt{\epsilon_{p_2q}}=\frac{1}{\sqrt2}(b_1\sqrt{p_2}+b_2\sqrt{q})$,  thus   $\epsilon_{p_2q}$ and $2\epsilon_{p_2q}$ are not squares in $\KK_1$, but $2p_2\epsilon_{p_2q}$ and $2q\epsilon_{p_2q}$ are.
   \item If $2p_2(a\pm1)$ is a square in $\NN$, then
$\left\{\begin{array}{rl}
    a\pm1 &=2b_1^2p_2\\
    a\mp1 &=2b_2^2q,
    \end{array}\right.$
  so $\sqrt{\epsilon_{p_2q}}=b_1\sqrt{p_2}+b_2\sqrt{q}$, thus   $\epsilon_{p_2q}$ is not a square in  $\KK_1$, but $p_2\epsilon_{p_2q}$ and $q\epsilon_{p_2q}$ are.
 \end{enumerate}

  2. Similarly, let $\epsilon_{p_1p_2q}=x+y\sqrt {p_1p_2q}$, where $x$ and $y$ are integers of different parities satisfying $x^2-1=p_1p_2qy^2$, hence $(x\pm1)(x\mp1)=p_1p_2qy^2$ and the gcd of $x\pm1$, $x\mp1$ divides  2. By  Lemma \ref{5}, $2(x\pm1)$ and $2p_1p_2q(x\pm1)$ are not squares in $\NN$. Thus, the decomposition uniqueness in $\ZZ$ enables us   to distinguish the following cases:
  \begin{enumerate}[\rm\indent i.]
    \item If $x\pm1$ is a square in $\NN$, then
$\left\{\begin{array}{rl}
    x\pm1 &=y_1^2\\
    x\mp1 &=y_2^2p_1p_2q,
    \end{array}\right.$\\
    so $\sqrt{2\epsilon_{p_1p_2q}}=y_1+y_2\sqrt{p_1p_2q}$, hence  $\epsilon_{p_1p_2q}$ is not a square in $\KK_1^+$, but $2\epsilon_{p_1p_2q}$ is.
    \item If $p_1(x\pm1)$ is a square in $\NN$, then
$\left\{\begin{array}{rl}
    x\pm1 &=y_1^2p_1\\
    x\mp1 &=y_2^2p_2q,
    \end{array}\right.$\\
  so $\sqrt{2\epsilon_{p_1p_2q}}=y_1\sqrt{p_1}+y_2\sqrt{p_2q}$, hence   $\epsilon_{p_1p_2q}$ is not a square in $\KK_1^+$, but  $2\epsilon_{p_1p_2q}$ is.
    \item If $2p_1(x\pm1)$ is a square in $\NN$, then
$\left\{\begin{array}{rl}
    x\pm1 &=2y_1^2p_1\\
    x\mp1 &=2y_2^2p_2q,
    \end{array}\right.$\\
 so $\sqrt{\epsilon_{p_1p_2q}}=y_1\sqrt{p_1}+y_2\sqrt{p_2q}$, hence   $\epsilon_{p_1p_2q}$ is a square in $\KK_1^+$.
    \item If $p_2(x\pm1)$ is a square in $\NN$, then  $\sqrt{2\epsilon_{p_1p_2q}}=y_1\sqrt{p_2}+y_2\sqrt{p_1q}$, hence   $\epsilon_{p_1p_2q}$, $2\epsilon_{p_1p_2q}$ are not squares in $\KK_1^+$, but  $2p_2\epsilon_{p_1p_2q}$ is.
    \item If $2p_2(x\pm1)$ is a square in $\NN$, then  $\sqrt{\epsilon_{p_1p_2q}}=y_1\sqrt{p_2}+y_2\sqrt{p_1q}$, hence   $\epsilon_{p_1p_2q}$ is not a square in $\KK_1^+$, but $p_2\epsilon_{p_1p_2q}$ is.
    \item If $q(x\pm1)$ is a square in $\NN$, then  $\sqrt{2\epsilon_{p_1p_2q}}=y_1\sqrt{q}+y_2\sqrt{p_1p_2}$, hence   $\epsilon_{p_1p_2q}$, $2\epsilon_{p_1p_2q}$ are not squares in $\KK_1^+$, but  $2q\epsilon_{p_1p_2q}$ is.
    \item If $2q(x\pm1)$ is a square in $\NN$, then $\sqrt{\epsilon_{p_1p_2q}}=y_1\sqrt{q}+y_2\sqrt{p_1p_2}$, hence   $\epsilon_{p_1p_2q}$  is not a square in  $\KK_1^+$, but $q\epsilon_{p_1p_2q}$ is.
  \end{enumerate}
 Consequently, if $a\pm1$ and  $\left(x\pm1\text{ or }p_1(x\pm1)\right)$  are  squares in $\NN$,  then $2\epsilon_{p_2q}$, $2\epsilon_{p_1p_2q}$ are squares in  $\KK_1^+$; hence  $\epsilon_{p_2q}\epsilon_{p_1p_2q}$ is a square in $\KK_1^+$. Thus  $\left\{\epsilon_{p_1},  \epsilon_{p_2q}, \sqrt{\epsilon_{p_2q}\epsilon_{p_1p_2q}}\right\}$ is a  $\mathrm{F}.\mathrm{S}.\mathrm{U}$ of $\KK_1^+$, and as $2\epsilon_{p_2q}$ is a square in $\KK_1^+$, so by Lemma \ref{3}\\ $\left\{\epsilon_{p_1},  \sqrt{i\epsilon_{p_2q}}, \sqrt{\epsilon_{p_2q}\epsilon_{p_1p_2q}}\right\}$ is a $\mathrm{F}.\mathrm{S}.\mathrm{U}$ de $\KK_1$. Thus $Q_{\KK_1}=2$.\\
 \indent The other cases are similarly treated.
\end{proof}
\subsection{\textbf{$\mathrm{F}.\mathrm{S}.\mathrm{U}$  of the field  $\KK_2$}}
As  $p_1$ and $p_2$ play symmetric roles, so we similarly get the $\mathrm{F}.\mathrm{S}.\mathrm{U}$ of
 $\KK_2=\kk(\sqrt{p_2})=\QQ(\sqrt{p_2}, \sqrt{p_1q}, i)$.
\begin{propo}\label{28}
Keep the previous notations and  put $\epsilon_{p_1q}=a+b\sqrt {p_1q}$.
\begin{enumerate}[\upshape\indent1.]
\item If $a\pm1$ and $(x\pm1$ or $p_2(x\pm1))$   are squares in $\NN$,  then $\left\{\epsilon_{p_2},  \epsilon_{p_1q}, \sqrt{\epsilon_{p_1q}\epsilon_{p_1p_2q}}\right\}$ is a  $\mathrm{F}.\mathrm{S}.\mathrm{U}$ of $\KK_2^+$ and  that of  $\KK_2$ is  $\left\{\epsilon_{p_2}, \sqrt{i\epsilon_{p_1q}}, \sqrt{\epsilon_{p_1q}\epsilon_{p_1p_2q}}\right\}$; moreover,   $Q_{\KK_2}=2$.
\item If $x\pm1$ or $p_2(x\pm1)$  is a square in  $\NN$ and $a\pm1$ is not, then $\left\{\epsilon_{p_2},  \epsilon_{p_1q}, \epsilon_{p_1p_2q}\right\}$ is a  $\mathrm{F}.\mathrm{S}.\mathrm{U}$ of $\KK_2^+$ and that of  $\KK_2$ is  $\left\{\epsilon_{p_2},  \epsilon_{p_1q}, \sqrt{i\epsilon_{p_1p_2q}}\right\}$; moreover,    $Q_{\KK_2}=2$.
\item If $a\pm1$ and $2p_2(x\pm1)$   are squares in $\NN$,  then $\left\{\epsilon_{p_2},  \epsilon_{p_1q}, \sqrt{\epsilon_{p_1p_2q}}\right\}$ is a  $\mathrm{F}.\mathrm{S}.\mathrm{U}$ of $\KK_2^+$ and that of  $\KK_2$ is  $\left\{\epsilon_{p_2}, \sqrt{i\epsilon_{p_1q}}, \sqrt{\epsilon_{p_1p_2q}}\right\}$; moreover,    $Q_{\KK_2}=2$.
\item If $2p_2(x\pm1)$  is a square in $\NN$ and  $a\pm1$ is not, then $\left\{\epsilon_{p_2},  \epsilon_{p_1q}, \sqrt{\epsilon_{p_1p_2q}}\right\}$ is a  $\mathrm{F}.\mathrm{S}.\mathrm{U}$ of both of $\KK_2^+$ and  $\KK_2$; moreover,   $Q_{\KK_2}=1$.
\item $\left\{\epsilon_{p_2},  \epsilon_{p_1q}, \epsilon_{p_1p_2q}\right\}$ is a  $\mathrm{F}.\mathrm{S}.\mathrm{U}$ of $\KK_2^+$,  $\left\{\epsilon_{p_2},  \sqrt{i\epsilon_{p_1q}}, \epsilon_{p_1p_2q}\right\}$ is a  $\mathrm{F}.\mathrm{S}.\mathrm{U}$ of $\KK_2$ and   $Q_{\KK_2}=2$ if one of the following assertions is satisfied:
\begin{enumerate}[\upshape\indent i.]
\item  $a\pm1$ and $(p_1(x\pm1)$ or $2p_1(x\pm1))$   are squares in $\NN$.
\item  $a\pm1$ and $(q(x\pm1)$ or $2q(x\pm1))$   are squares in $\NN$.
\end{enumerate}
 \item $\left\{\epsilon_{p_2},  \epsilon_{p_1q}, \sqrt{\epsilon_{p_1q}\epsilon_{p_1p_2q}}\right\}$ is a  $\mathrm{F}.\mathrm{S}.\mathrm{U}$ of both of $\KK_2^+$ and  $\KK_2$, and   $Q_{\KK_2}=1$,   if one of the following assertions is satisfied:
\begin{enumerate}[\upshape\indent i.]
\item  $p_1(a\pm1)$ and $(p_1(x\pm1)$ or $q(x\pm1))$   are squares in $\NN$.
\item  $2p_1(a\pm1)$ and $(2p_1(x\pm1)$ or $2q(x\pm1))$   are squares in $\NN$.
\end{enumerate}
 \item $\left\{\epsilon_{p_2},  \epsilon_{p_1q}, \epsilon_{p_1p_2q}\right\}$ is a  $\mathrm{F}.\mathrm{S}.\mathrm{U}$ of $\KK_2^+$,  $\left\{\epsilon_{p_2},  \epsilon_{p_1q}, \sqrt{i\epsilon_{p_1q}\epsilon_{p_1p_2q}}\right\}$ is a  $\mathrm{F}.\mathrm{S}.\mathrm{U}$  of $\KK_2$ and   $Q_{\KK_2}=2$ if one of the following assertions is satisfied:
\begin{enumerate}[\upshape\indent i.]
\item  $p_1(a\pm1)$ and $(2p_1(x\pm1)$ or $2q(x\pm1))$   are squares in $\NN$.
\item  $2p_1(a\pm1)$ and $(p_1(x\pm1)$ or $q(x\pm1))$   are squares in $\NN$.
\end{enumerate}
\end{enumerate}
\end{propo}
\subsection{\textbf{$\mathrm{F}.\mathrm{S}.\mathrm{U}$  of the field  $\KK_3$}}
Let $\KK_3=\kk(\sqrt{q})=\QQ(\sqrt{q}, \sqrt{p_1p_2}, i)$.
\begin{propo}\label{29}
Keep the previous notations and  assume  $N(\epsilon_{p_1p_2})=1$.  Then   $Q_{\KK_3}=2$ and we have:
\begin{enumerate}[\upshape\indent1.]
\item If  $2q(x\pm1)$ is a square in $\NN$,  then $\left\{\epsilon_{q},  \epsilon_{p_1p_2}, \sqrt{\epsilon_{p_1p_2q}}\right\}$ is a $\mathrm{F}.\mathrm{S}.\mathrm{U}$ of $\KK_3^+$ and that of   $\KK_3$ is $\left\{\epsilon_{p_1p_2}, \sqrt{\epsilon_{p_1p_2q}}, \sqrt{i\epsilon_{q}}\right\}$.
\item If $x\pm1$ or $q(x\pm1)$ is a square in $\NN$, then $\left\{\epsilon_{q},  \epsilon_{p_1p_2}, \sqrt{\epsilon_{q}\epsilon_{p_1p_2q}}\right\}$ is a $\mathrm{F}.\mathrm{S}.\mathrm{U}$ of $\KK_3^+$ and that of   $\KK_3$ is a $\left\{\epsilon_{p_1p_2}, \sqrt{\epsilon_{q}\epsilon_{p_1p_2q}}, \sqrt{i\epsilon_{q}}\right\}$.
\item If
  $p_1(x\pm1)$  or $p_2(x\pm1)$ is a  square in $\NN$, then  $\left\{\epsilon_{q},  \epsilon_{p_1p_2}, \sqrt{\epsilon_{q}\epsilon_{p_1p_2}\epsilon_{p_1p_2q}}\right\}$ is a $\mathrm{F}.\mathrm{S}.\mathrm{U}$  of  $\KK_3^+$ and that of $\KK_3$ is  $\left\{\epsilon_{p_1p_2}, \sqrt{\epsilon_{q}\epsilon_{p_1p_2}\epsilon_{p_1p_2q}}, \sqrt{i\epsilon_{q}} \right\}$.
\item If
  $2p_1(x\pm1)$  or
 $2p_2(x\pm1)$ is a  square $\NN$, then $\left\{\epsilon_{q},  \epsilon_{p_1p_2}, \sqrt{\epsilon_{p_1p_2}\epsilon_{p_1p_2q}}\right\}$ is a $\mathrm{F}.\mathrm{S}.\mathrm{U}$  of  $\KK_3^+$ and that of $\KK_3$ is $\left\{\epsilon_{p_1p_2}, \sqrt{\epsilon_{p_1p_2}\epsilon_{p_1p_2q}}, \sqrt{i\epsilon_{q}} \right\}$.
\end{enumerate}
\end{propo}
\begin{proof}
  As the norms of $\epsilon_{q}$,  $\epsilon_{p_1p_2}$ and $\epsilon_{p_1p_2q}$ are equal to  $1$, then a $\mathrm{F}.\mathrm{S}.\mathrm{U}$ of $\KK_3^+$ is a system consisting of three elements chosen from $B'$, where  $B'=B\cup\left\{\sqrt \mu/\mu\in B\ et\ \sqrt \mu \in\KK^+\right\}$, with  $$B=\left\{\epsilon_{q}, \epsilon_{p_1p_2}, \epsilon_{p_1p_2q}, \epsilon_q\epsilon_{p_1p_2}, \epsilon_q\epsilon_{p_1p_2q}, \epsilon_{p_1p_2}\epsilon_{p_1p_2q}, \epsilon_q\epsilon_{p_1p_2}\epsilon_{p_1p_2q}\right\}.$$
According to Lemma \ref{1:048},  $\epsilon_q$ is not a square in $\QQ(\sqrt{q})$, but $2\epsilon_q$ is.\\
 \indent Put  $\epsilon_{p_1p_2}=a+b\sqrt{p_1p_2}$, then $a^2-1=b^2p_1p_2$. Hence by Lemmas \ref{5} and \ref{3:105} we get that only $2p_1(a\pm1)$ is a square in $\NN$,  so   $\epsilon_{p_1p_2}$ is not a square in $\KK_3^+$, but $p_1\epsilon_{p_1p_2}$, $p_2\epsilon_{p_1p_2}$ are.\\
\indent Let $\epsilon_{p_1p_2q}=x+y\sqrt{p_1p_2q}$, then proceeding as in the proof of Proposition \ref{27},  we get:
\begin{enumerate}[\upshape\indent a.]
\item If $x\pm1$ is a square in $\NN$, then $\sqrt{2\epsilon_{p_1p_2q}}=y_1+y_2\sqrt{p_1p_2q}$,  so   $\epsilon_{p_1p_2q}$ is not a square in $\KK_3^+$, but $2\epsilon_{p_1p_2q}$ is.
\item If $p_1(x\pm1)$ is a square in $\NN$, then  $\sqrt{2\epsilon_{p_1p_2q}}=y_1\sqrt{p_1}+y_2\sqrt{p_2q}$,  so    $\epsilon_{p_1p_2q}$ and $2\epsilon_{p_1p_2q}$ are not squares in $\KK_3^+$, but $2p_1\epsilon_{p_1p_2q}$ and $2p_2q\epsilon_{p_1p_2q}$ are.
\item If $2p_1(x\pm1)$ is a square in $\NN$, then $\sqrt{\epsilon_{p_1p_2q}}=y_1\sqrt{p_1}+y_2\sqrt{p_2q}$,  so    $\epsilon_{p_1p_2q}$ is not a square in $\KK_3^+$, but $p_1\epsilon_{p_1p_2q}$ and $p_2q\epsilon_{p_1p_2q}$ are.
\item If $p_2(x\pm1)$ is a square in $\NN$,  then $\sqrt{2\epsilon_{p_1p_2q}}=y_1\sqrt{p_2}+y_2\sqrt{p_1q}$,  so    $\epsilon_{p_1p_2q}$ and $2\epsilon_{p_1p_2q}$ are not squares in $\KK_3^+$, but $2p_2\epsilon_{p_1p_2q}$ and $2p_1q\epsilon_{p_1p_2q}$ are.
\item If $2p_2(x\pm1)$ is a square in $\NN$,  then $\sqrt{\epsilon_{p_1p_2q}}=y_1\sqrt{p_2}+y_2\sqrt{p_1q}$,  so    $\epsilon_{p_1p_2q}$ is not a square in $\KK_3^+$, but $p_2\epsilon_{p_1p_2q}$, $p_1q\epsilon_{p_1p_2q}$ are.
\item If $q(x\pm1)$ is a square in $\NN$, then $\sqrt{2\epsilon_{p_1p_2q}}=y_1\sqrt{q}+y_2\sqrt{p_1p_2}$,  so    $\epsilon_{p_1p_2q}$ is not a square in $\KK_3^+$, but  $2\epsilon_{p_1p_2q}$ is.
\item If $2q(x\pm1)$ is a square in $\NN$, then $\sqrt{\epsilon_{p_1p_2q}}=y_1\sqrt{q}+y_2\sqrt{p_1p_2}$,  so    $\epsilon_{p_1p_2q}$  is a square in $\KK_3^+$.
 \end{enumerate}
  Consequently, we have:
 \begin{enumerate}[\upshape\indent 1.]
\item If $x\pm1$ or $q(x\pm1)$ is a square in $\NN$, then $\epsilon_q\epsilon_{p_1p_2q}$ is a square in $\KK_3^+$,  so  \{$\epsilon_q$, $\epsilon_{p_1p_2}$, $\sqrt{\epsilon_q\epsilon_{p_1p_2q}}$\} is a $\mathrm{F}.\mathrm{S}.\mathrm{U}$ of $\KK_3^+$. As $2\epsilon_q$ is a square in $\KK_3^+$,  then by Lemma \ref{3},  \{$\epsilon_{p_1p_2}$, $\sqrt{\epsilon_q\epsilon_{p_1p_2q}}$, $\sqrt{i\epsilon_q}$\} is a $\mathrm{F}.\mathrm{S}.\mathrm{U}$ of $\KK_3$.
\item If  $2q(x\pm1)$ is a square in $\NN$, then $\epsilon_{p_1p_2q}$ is a square in $\KK_3^+$, hence $\left\{\epsilon_{q},  \epsilon_{p_1p_2}, \sqrt{\epsilon_{p_1p_2q}}\right\}$ is a $\mathrm{F}.\mathrm{S}.\mathrm{U}$  of  $\KK_3^+$, and by Lemma \ref{3}  $\left\{\epsilon_{p_1p_2}, \sqrt{\epsilon_{p_1p_2q}}, \sqrt{i\epsilon_{q}}\right\}$ is a $\mathrm{F}.\mathrm{S}.\mathrm{U}$ of $\KK_3$.
\item If   $p_1(x\pm1)$  is a square in $\NN$, then  $2p_1\epsilon_{p_1p_2q}$ is a square in $\KK_3^+$, and since  $p_1\epsilon_{p_1p_2}$ and  $2\epsilon_{q}$ are too,  so  $\epsilon_{q}\epsilon_{p_1p_2}\epsilon_{p_1p_2q}$ is a square in $\KK_3^+$, hence  $\left\{\epsilon_{q},  \epsilon_{p_1p_2}, \sqrt{\epsilon_{q}\epsilon_{p_1p_2}\epsilon_{p_1p_2q}}\right\}$ is a $\mathrm{F}.\mathrm{S}.\mathrm{U}$  of  $\KK_3^+$. As $2\epsilon_q$ is a square in $\KK_3^+$,  then by Lemma \ref{3}\\ $\left\{\epsilon_{p_1p_2}, \sqrt{\epsilon_{q}\epsilon_{p_1p_2}\epsilon_{p_1p_2q}}, \sqrt{i\epsilon_{q}} \right\}$ is a $\mathrm{F}.\mathrm{S}.\mathrm{U}$  of  $\KK_3$.
\item If   $2p_2(x\pm1)$  is a square in $\NN$, then   $p_2\epsilon_{p_1p_2q}$ is a square in $\KK_3^+$, and  since $p_2\epsilon_{p_1p_2}$ is too, so $\epsilon_{p_1p_2}\epsilon_{p_1p_2q}$ is a square in $\KK_3^+$. This yields that $\left\{\epsilon_{q},  \epsilon_{p_1p_2}, \sqrt{\epsilon_{p_1p_2}\epsilon_{p_1p_2q}}\right\}$ is a $\mathrm{F}.\mathrm{S}.\mathrm{U}$  of  $\KK_3^+$. On the other hand,   $2\epsilon_q$ is a square in $\KK_3^+$,  then by Lemma \ref{3}, $\left\{\epsilon_{p_1p_2}, \sqrt{\epsilon_{p_1p_2}\epsilon_{p_1p_2q}}, \sqrt{i\epsilon_{q}} \right\}$ is a $\mathrm{F}.\mathrm{S}.\mathrm{U}$  of  $\KK_3$.
\end{enumerate}
 The other cases are similarly proved.
\end{proof}
\begin{propo}\label{36}
Keep the previous notations and  assume  $N(\epsilon_{p_1p_2})=-1$.  Then   $Q_{\KK_3}=2$ and we have:
\begin{enumerate}[\upshape\indent1.]
\item If  $2q(x\pm1)$ is a square in $\NN$,  then $\left\{\epsilon_{q},  \epsilon_{p_1p_2}, \sqrt{\epsilon_{p_1p_2q}}\right\}$ is a $\mathrm{F}.\mathrm{S}.\mathrm{U}$ of $\KK_3^+$ and that of   $\KK_3$ is $\left\{\epsilon_{p_1p_2}, \sqrt{\epsilon_{p_1p_2q}}, \sqrt{i\epsilon_{q}}\right\}$.
\item If $x\pm1$ or $q(x\pm1)$ is a square in $\NN$, then $\left\{\epsilon_{q},  \epsilon_{p_1p_2}, \sqrt{\epsilon_{q}\epsilon_{p_1p_2q}}\right\}$ is a $\mathrm{F}.\mathrm{S}.\mathrm{U}$  of $\KK_3^+$ and that of   $\KK_3$ is $\left\{\epsilon_{p_1p_2}, \sqrt{\epsilon_{q}\epsilon_{p_1p_2q}}, \sqrt{i\epsilon_{q}}\right\}$.
\item In the other cases $\left\{\epsilon_{q},  \epsilon_{p_1p_2}, \epsilon_{p_1p_2q}\right\}$ is a $\mathrm{F}.\mathrm{S}.\mathrm{U}$ of $\KK_3^+$ and that of   $\KK_3$ is $\left\{\epsilon_{p_1p_2}, \epsilon_{p_1p_2q}, \sqrt{i\epsilon_{q}}\right\}$.
\end{enumerate}
\end{propo}
\begin{proof}
As $N(\epsilon_{p_1p_2}1)=-$, so  by  Lemma \ref{1:048}, only $\epsilon_{p_1p_2q}$ and $\epsilon_{q}\epsilon_{p_1p_2q}$ can be squares in  $\KK_3^+$. Proceeding as above, we get the results.
\end{proof}
\section{\textbf{The  ambiguous classes of $\kk/\QQ(i)$}}
 Let $F=\QQ(i)$ and $\kk=\QQ(\sqrt{p_1p_2q}, i)$. We denote by  $\mathrm{A}m(\kk/F)$  the group of the ambiguous classes of $\kk/F$ and by $\mathrm{A}m_s(\kk/F)$ the subgroup of $\mathrm{A}m(\kk/F)$ generated by  the strongly ambiguous classes.
As  $p_1\equiv p_2\equiv1\pmod4$, so there exist $e$, $f$, $g$ and $h$ in $\NN$ such that  $p_1=e^2+4f^2=\pi_1\pi_2$ and $p_2=g^2+4h^2=\pi_3\pi_4$. Put  $\pi_1=e+2if$,  $\pi_2=e-2if$, $\pi_3=g+2ih$ and  $\pi_4=g-2ih$. Let  $\mathcal{H}_j$ (resp. $\mathcal{Q}$)  be the prime ideal of  $\kk$ above $\pi_j$ (resp. $q$), where $j\in\{1, 2, 3, 4\}$. It is easy to see  that $\mathcal{H}_j^2=(\pi_j)$ and $\mathcal{Q}^2=(q)$. Therefore $[\mathcal{Q}]$ and  $[\mathcal{H}_j]$ are in $\mathrm{A}m_s(\kk/F)$, for all  $j\in\{1, 2, 3, 4\}$.  Keep the notation $\epsilon_{p_1p_2q}=x+y\sqrt{p_1p_2q}$. In this section,  we will determine  generators of $\mathrm{A}m_s(\kk/F)$ and $\mathrm{A}m(\kk/F)$.  Let us first prove the following result.
\begin{lem}\label{7}
Consider the prime ideals $\mathcal{H}_j$ of $\kk$, $1\leq j\leq4$.
\begin{enumerate}[\rm\indent1.]
  \item If $x\pm1$ is  a square in $\NN$, then  $\left|\langle[\h], [\hh], [\hhh], [\hhhh]\rangle\right|=16$.
     \item If $q(x\pm1)$ or $2q(x\pm1)$ is  a square in $\NN$, then $\left|\langle[\h], [\hh],  [\hhh]\rangle\right|=$ \\$\left|\langle[\h], [\hhh],  [\hhhh]\rangle\right|=8$.
   \item If $p_1(x\pm1)$ or $2p_1(x\pm1)$ is  a square in $\NN$, then $\left|\langle[\h], [\hhh],  [\hhhh]\rangle\right|=8$.
    \item If $p_2(x\pm1)$ or $2p_2(x\pm1)$ is  a square in $\NN$, then  $\left|\langle[\h], [\hh],  [\hhh]\rangle\right|=8$.
\end{enumerate}
\end{lem}
\begin{proof}
Since $\mathcal{H}_j^2=(\pi_j)$, for all $1\leq j\leq2$, and since also  $\sqrt{e^2+(2f)^2}=\sqrt{p_1}\not\in\QQ(\sqrt{p_1p_2q})$ and $\sqrt{g^2+(2h)^2}=\sqrt{p_2}\not\in\QQ(\sqrt{p_1p_2q})$, so,  according to \cite[Proposition 1]{AZT12-2},  $\mathcal{H}_j$ are not  principal  in $\kk$.\par
1. If $x\pm1$ is  a square in $\NN$, then ${\epsilon_{p_1p_1q}}$ is not a $\mathrm{F}.\mathrm{S}.\mathrm{U}$ of $\kk$ (by Lemma \ref{6}) and for all prime $\ell$ dividing $p_1p_2q$,  $\ell(x+1)$,  $\ell(x-1)$,  $2\ell(x+1)$,  $2\ell(x-1)$  are not squares in $\NN$. We have:

  $(\h\hh)^2=(p_1)$,  $(\hhh\hhhh)^2=(p_2)$ and $\mathcal{Q}^2=(q)$, hence according to \cite[Proposition 2]{AZT12-2},  $\h\hh$, $\hhh\hhhh$ and $\mathcal{Q}$ are not principal in $\kk$.

For $i\in\{1, 2\}$ and $j\in\{3, 4\}$,   $\mathcal{H}_i\mathcal{H}_j$ is not  principal in $\kk$, in fact,   $(\mathcal{H}_i\mathcal{H}_j)^2=(\pi_i\pi_j)$ and as
 $\pi_1\pi_3=(eg-4fh)+2i(eh+fg)$,
   $\pi_1\pi_4=(eg+4fh)-2i(eh-fg)$,
   $\pi_2\pi_3=(eg+4fh)+2i(eh-fg)$,
   $\pi_2\pi_4=(eg-4fh)-2i(eh+fg)$,
and also
 $\sqrt{(eg-4fh)^2+4(eh+fg)^2}=\sqrt{(eg+4fh)^2+4(eh-fg)^2}=\sqrt{p_1p_2}\not\in\qq$, hence \cite[Proposition 1]{AZT12-2} yields the result.

For $i\in\{1, 2\}$ (resp. $j\in\{3, 4\}$), the ideal   $\mathcal{H}_i\mathcal{H}_3\mathcal{H}_4$ (resp. $\mathcal{H}_1\mathcal{H}_2\mathcal{H}_j$) is not  principal, since  $(\mathcal{H}_i\mathcal{H}_3\mathcal{H}_4)^2=(\pi_ip_2)$ (resp. $(\mathcal{H}_1\mathcal{H}_2\mathcal{H}_j)^2=(p_1\pi_j)$), and as $\sqrt{((p_2e)^2+4(p_2f)^2)}=p_2\sqrt{p_1}\not\in\qq$ (resp. $\sqrt{((p_1g)^2+4(p_1h)^2)}=p_1\sqrt{p_2}\not\in\qq$), then \cite[Proposition 1]{AZT12-2} states the result.

Finely, as $(\h\hh\hhh\hhhh)^2=(p_1p_2)$, so \cite[Remark 1]{AZT12-2} implies that $\h\hh\hhh\hhhh$ is not principal in $\kk$. Note at the end that $[\h\hh\hhh\hhhh]=[\mathcal{Q}]$, since $[\h\hh\hhh\hhhh\mathcal{Q}]=[(\sqrt{p_1p_2q})]=1$.

2. If $q(x\pm1)$ or $2q(x\pm1)$ are not squares  in $\NN$, then  according to \cite[Proposition 2]{AZT12-2} $[\h\hh\hhh\hhhh]=[\mathcal{Q}]$ is  principal in $\kk$. Hence the result.

The other assertions are similarly proved.
\end{proof}

 Determine now  generators of $ \mathrm{A}m_s(\kk/F)$ and $\mathrm{A}m(\kk/F)$. According to the ambiguous class number formula (see \cite{Ch-33}),
the genus number, $[(\kk/F)^*:\kk]$, is given by:
 \begin{equation}\label{51}
|\mathrm{A}m(\kk/F)|=[(\kk/F)^*:\kk]=\frac{h(F)2^{t-1}}{[E_F: E_F\cap N_{\kk/F}(\kk^\times)]},
\end{equation}
where $h(F)$ is the class number of $F$ and $t$ is the number of finite and infinite primes of $F$ ramified in $\kk/F$. Moreover as the class number
of  $F$ is equal to $1$, so the formula \eqref{51} yields that
 \begin{equation}\label{56}|\mathrm{A}m(\kk/F)|=[(\kk/F)^*:\kk]=2^r,\end{equation}
 where $r=\text{rank}\mathbf{C}l_2(\mathds{k})=t-e-1$ and $2^e=[E_F: E_F\cap N_{\kk/F}(\kk^\times)]$ (see for example \cite{McPaRa-95}).
  The relation between  $|\mathrm{A}m(\kk/F)|$ and $|\mathrm{A}m_s(\kk/F)|$ is given by the following formula (see for example \cite{Lem-13}):
\begin{equation}\label{50}
\frac{|\mathrm{A}m(\kk/F)|}{|\mathrm{A}m_s(\kk/F)|}=[E_F\cap N_{\kk/F}(\kk^\times):N_{\kk/F}(E_\kk)].
\end{equation}
 To continue, we need  the following lemma.
\begin{lem}[\cite{McPaRa-95}]\label{57}
Let $p_1\equiv p_2\equiv-q\equiv1\pmod4$ be different primes,  $F=\QQ(i)$ and $\kk=\QQ(\sqrt{p_1p_2q}, i)$.
\begin{enumerate}[\rm\indent1.]
  \item If $p_1\equiv p_2\equiv1 \pmod8$, then $i$ is  a norm in $\kk/F$.
  \item If $p_1\equiv5$ or $p_2\equiv 5\pmod8$, then $i$ is not a norm in $\kk/F$.
\end{enumerate}
\end{lem}
\begin{propo}\label{248}
Let $(\kk/F)^*$ denote the relative genus field of $\kk/F$.
\begin{enumerate}[\rm\indent1.]
\item $\k\varsubsetneq (\kk/F)^*$ and $[(\kk/F)^*:\k]=2\text{ or } 4$.
\item Assume  $p_1\equiv p_2\equiv1 \pmod8$.
\begin{enumerate}[\rm i.]
  \item If $x\pm1$ is a square in  $\NN$, then\\ $\mathrm{Am}(\kk/\QQ(i))=\mathrm{Am}_s(\kk/\QQ(i))=\langle[\h], [\hh], [\hhh], [\hhhh]\rangle.$
  \item Else, there  exist an unambiguous ideal $\mathcal{I}$ in $\kk/\QQ(i)$ of order  $2$ such that
    $\mathrm{Am}_s(\kk/\QQ(i))=\langle[\h], [\hh], [\hhh]\rangle\text{ or }\langle[\h], [\hhh], [\hhhh]\rangle$ and \\ $\mathrm{Am}(\kk/\QQ(i))=\langle[\h], [\hh], [\hhh], [\mathcal{I}]\rangle\text{ or }\langle[\h], [\hhh], [\hhhh], [\mathcal{I}]\rangle.$
\end{enumerate}
\item Assume  $p_1\equiv5$ or $p_2\equiv5\pmod8$, then neither $x + 1$ nor $x - 1$ is a square in $\NN$ and $\mathrm{Am}(\kk/\QQ(i))=\mathrm{Am}_s(\kk/\QQ(i))=\langle[\h], [\hh], [\hhh]\rangle$ or $\langle[\h], [\hhh], [\hhhh]\rangle$.
\end{enumerate}
\end{propo}
\begin{proof}
1. As $\kk=\QQ(\sqrt{p_1p_2q}, i)$, so $[\k:\kk]=4$. Moreover, according to \cite[Proposition 2, p. 90]{McPaRa-95}, $r=\text{rank}\mathbf{C}l_2(\mathds{k})=4$ if $p_1\equiv p_2\equiv1 \pmod8$ and $r=\text{rank}\mathbf{C}l_2(\mathds{k})=3$ if $p_1\equiv 5$ or $p_2\equiv5 \pmod8$, so $[(\kk/F)^*:\kk]=8 \text{ or } 16$. Hence $[(\kk/F)^*:\k]=2 \text{ or } 4$, and the   result derived.\par
2. Assume  $p_1\equiv p_2\equiv1 \pmod8$, hence $i$ is a norm in $\kk/\QQ(i)$ (Lemma \ref{57}), thus Formula \eqref{50} yields that
    \begin{align*}\dfrac{|\mathrm{Am}(\kk/\QQ(i))|}{|\mathrm{Am}_s(\kk/\QQ(i))|}&=[E_{\QQ(i)}\cap N_{\kk/\QQ(i)}(\kk^{\times}):N_{\kk/\QQ(i)}(E_\kk)]\\
                                                                               &=\left\{\begin{array}{ll}
                                                                                 1 \text{ if  $x\pm1$ is a square in } \NN,\\
                                                                                 2 \text{ if not, }
                                                                                 \end{array}\right.\end{align*}
since in the case where $x\pm1$ is a square in  $\NN$, we have $E_\kk=\langle i, \sqrt{i\epsilon_{p_1p_2q}}\rangle$, hence $[E_{\QQ(i)}\cap
N_{\kk/\QQ(i)}(\kk^{\times}):N_{\kk/\QQ(i)}(E_\kk)]=[< i >: < i >]=1$, and if not we have  $E_\kk=\langle i, \epsilon_{p_1p_2q}\rangle$, hence $[E_{\QQ(i)}\cap
N_{\kk/\QQ(i)}(\kk^{\times}):N_{\kk/\QQ(i)}(E_\kk)]=[< i >: < -1>]=2$.\\
\indent On the other hand, as  $p_1\equiv p_2\equiv1 \pmod8$, so   according to \cite[Proposition 2, p. 90]{McPaRa-95},  $r=4$. Therefore     $|\mathrm{Am}(\kk/\QQ(i))|=2^4.$\\
\indent i. If $x\pm1$ is a square in $\NN$, then $\mathrm{Am}_s(\kk/\QQ(i))=\mathrm{Am}(\kk/\QQ(i))$, hence by Lemma \ref{7} we get  $\mathrm{Am}(\kk/\QQ(i))=\mathrm{Am}_s(\kk/\QQ(i))=\langle[\h], [\hh], [\hhh], [\hhhh]\rangle.$\\
\indent ii.  If $x+1$ and $x-1$ are not squares in $\NN$, then $$|\mathrm{Am}(\kk/\QQ(i))|=2|\mathrm{Am}_s(\kk/\QQ(i))|=16,$$  hence Lemma \ref{7} yields that $\mathrm{Am}_s(\kk/\QQ(i))=\langle[\h], [\hh], [\hhh]\rangle$ or $\langle[\h], [\hhh], [\hhhh]\rangle.$\\
\indent Consequently,  there exist an unambiguous  ideal $\mathcal{I}$ in $\kk/F$ of order $2$ such that
    $$\mathrm{Am}(\kk/\QQ(i))= \langle[\h], [\hh], [\hhh], [\mathcal{I}]\rangle\text{ or }\langle[\h], [\hhh], [\hhhh], [\mathcal{I}]\rangle.$$
    By Chebotarev theorem,  $\mathcal{I}$ can  always be chosen  as a prime ideal of $\kk$ above a prime $l$ in $\QQ$, which splits completely in $\kk$.

3. Assume   $p_1\equiv5$ or  $p_2\equiv5\pmod8$, hence $i$ is not a norm in $\kk/\QQ(i)$ (Lemma \ref{57}) and $x+1$, $x-1$ are not squares in $\NN$, for if $x\pm1$ is a square in $\NN$, then the Legendre symbol implies that $$1=\left(\frac{x\pm1}{p_j}\right)=\left(\frac{x\mp1\pm2}{p_j}\right)=\left(\frac{2}{p_j}\right)\ \text{ for all }j\in\{1, 2\},$$ which is absurd.  Thus $|\mathrm{Am}(\kk/\QQ(i))|=2^3$ and
    $$\dfrac{|\mathrm{Am}(\kk/\QQ(i))|}{|\mathrm{Am}_s(\kk/\QQ(i))|}=[E_{\QQ(i)}\cap N_{\kk/\QQ(i)}(\kk^{\times}):N_{\kk/\QQ(i)}(E_\kk)]=1.$$                                                                         Hence by Lemma \ref{7} we get
       $\mathrm{Am}(\kk/\QQ(i))=\mathrm{Am}_s(\kk/\QQ(i))=\langle[\h], [\hh], [\hhh]\rangle$ or $\langle[\h], [\hhh], [\hhhh]\rangle.$
      This completes the proof.
\end{proof}
\section{\bf{Capitulation}}
In this section,  we will determine the classes of  $\mathbf{C}l_2(\kk)$, the  $2$-class group of $\kk$, that capitulate in  $\KK_j$, for all $j\in\{1, 2, 3\}$. For this we need the following theorem.
\begin{The}[\cite{HS82}]\label{1}
 Let $K/k$ be a cyclic extension of prime degree, then  the number of classes that capitulate in $K/k$ is:
 $[K:k][E_k:N_{K/k}(E_K)],$
 where $E_k$ and $E_K$ are the  unit groups of $k$ and $K$ respectively.
\end{The}
\begin{The}\label{226}
Let $\KK_j$, $1\leq j\leq3$, be the three unramified quadratic extensions of $\kk$ defined above.
\begin{enumerate}[\rm\indent1.]
\item  Let $\epsilon_{p_2q}=a+b\sqrt{p_2q}$.
\begin{enumerate}[\rm i.]
  \item If  $x\pm1$ is a square in $\NN$ and $a+1$, $a-1$ are not, then $|\kappa_{\KK_1}|=8$.
  \item If  $a\pm1$ and $(2p_1(x\pm1)$ or $p_2(x\pm1))$ are  squares in $\NN$,  then $|\kappa_{\KK_1}|=2$.
  \item For the other cases $|\kappa_{\KK_1}|=4$.
\end{enumerate}
\item  Let $\epsilon_{p_1q}=a+b\sqrt{p_1q}$.
\begin{enumerate}[\rm i.]
  \item If  $x\pm1$ is a square in $\NN$ and $a+1$, $a-1$ are not, then $|\kappa_{\KK_2}|=8$.
  \item If  $a\pm1$ and $(2p_1(x\pm1)$ or $p_2(x\pm1))$ are  squares in $\NN$,  then $|\kappa_{\KK_2}|=2$.
  \item For the other cases $|\kappa_{\KK_2}|=4$.
\end{enumerate}
\item Let $\epsilon_{p_1p_2}=a+b\sqrt{p_1p_2}$.
\begin{enumerate}[\rm i.]
\item If  $N(\epsilon_{p_1p_2})=1$, then
\begin{enumerate}[\rm a.]
  \item If  $x\pm1$ is a square in $\NN$, then $|\kappa_{\KK_3}|=4$.
  \item Else  $|\kappa_{\KK_3}|=2$.
\end{enumerate}
\item If  $N(\epsilon_{p_1p_2})=-1$, then
\begin{enumerate}[\rm a.]
  \item If  $q(x\pm1)$ or $2q(x\pm1)$ is a square in $\NN$, then $|\kappa_{\KK_3}|=2$.
  \item Else  $|\kappa_{\KK_3}|=4$.
\end{enumerate}
\end{enumerate}
\end{enumerate}
\end{The}
\begin{proof}
Note first that, by Lemma \ref{6}, $E_{\kk}=\langle i, \sqrt{i\epsilon_{p_1p_2q}}\rangle$ if $x\pm1$ is a square in $\NN$,  and $E_{\kk}=\langle i, \epsilon_{p_1p_2q}\rangle$ otherwise.

$1.$ i. According to   Proposition $\ref{27}$, if $x\pm1$ is a square in $\NN$ and $a+1$, $a-1$ are not, then $N_{\KK_1/\kk}(E_{\KK_1})=\langle i, \epsilon_{p_1p_2q}\rangle$. Hence Theorem \ref{1} yields that $|\kappa_{\KK_1}|=8$.

ii. if $a\pm1$ and $(2p_1(x\pm1)$ or $p_2(x\pm1))$ are  squares in $\NN$,  then by Proposition $\ref{27}(1.\text{ and }3.)$ we get $N_{\KK_1/\kk}(E_{\KK_1})=\langle -1, i\epsilon_{p_1p_2q}\rangle$ or $\langle i, \epsilon_{p_1p_2q}\rangle$. Hence Theorem \ref{1} yields that $|\kappa_{\KK_1}|=2$.

iii. If $a\pm1$ and $x\pm1$ are squares in $\NN$, then $E_{\kk}=\langle i, \sqrt{i\epsilon_{p_1p_2q}}\rangle$ and $N_{\KK_1/\kk}(E_{\KK_1})=\langle i, \epsilon_{p_1p_2q}\rangle$. Hence Theorem \ref{1} yields that $|\kappa_{\KK_1}|=4$. \\
For the other cases, we have $E_{\kk}=\langle i, \epsilon_{p_1p_2q}\rangle$ and $N_{\KK_1/\kk}(E_{\KK_1})=\langle i, \epsilon_{p_1p_2q}^2\rangle$, $\langle -1, \epsilon_{p_1p_2q}\rangle$ or $\langle -1, i\epsilon_{p_1p_2q}\rangle$. Hence Theorem \ref{1} yields that $|\kappa_{\KK_1}|=4$.

The other assertions of the theorem are similarly proved.
\end{proof}
\subsection{Capitulation in $\KK_1$}
We begin this subsection by the following result.
\begin{propo}\label{9}
Let $d$ be a square-free integer and $p\equiv1\pmod4$ a prime divisor of $d$. Put  $k=\QQ(\sqrt d, i)$ and $p=\pi\pi'$, where $\pi$ and $\pi'$ are in $\QQ(i)$. Let $\mathcal{H}$ be a prime ideal of $k$ above $\pi$, then $\mathcal{H}$ capitulates in $K=k(\sqrt p)$.
\end{propo}
\begin{proof}
It is easy to see that   $\mathcal{H}$ ramifies in $k/\QQ(i)$ and it is of order $2$.
As  $\epsilon_{p}=\frac{1}{2}(x+y\sqrt{p})$ it is of norm  $-1$, so $x^2+4=y^2p$, hence by the decomposition uniqueness    there exist   $y_1$,  $y_2$ in $\ZZ[i]$ such that
 $$(1) \left\{\begin{array}{ll}
x\pm2i &=y_1^2\pi\\
x\mp2i &= y_2^2\pi',
\end{array}\right. \text{ or }
 (2) \left\{\begin{array}{ll}
x\pm2i &=iy_1^2\pi\\
x\mp2i &=-iy_2^2\pi'.
\end{array}\right. \text{ with }p=\pi\pi',\ y=y_1y_2.$$

The  system (1) implies that  $2x=y_1^2\pi+y_2^2\pi'$. Put
$\alpha = \frac{1}{2}(y_1\pi+y_2\sqrt p)$ and   $\beta =\frac{1}{2}(
y_2\pi'+y_1\sqrt p)$. Then $\alpha$ and $\beta$ are in $K=k(\sqrt
p)$, and we have:
 \begin{align*}
  \alpha^2 & = \frac{1}{4}(y_1^2\pi^2+y_2^2p+2y_1y_2\pi\sqrt p)\\
           &= \frac{1}{4}\pi(y_1^2\pi+y_2^2\pi'+2y\sqrt p),\ \text{ since }\  p=\pi\pi'\  \text{ and }\ y=y_1y_2.\\
           &= \frac{1}{4}\pi(2x+2y\sqrt p),\ \text{ since }\  2x=y_1^2\pi+y_2^2\pi'.\\
           &= \pi\epsilon_{p},\  \text{ since }\  \epsilon_{p}=\frac{1}{2}(x+y\sqrt p).
  \end{align*}
And as $\epsilon_{p}$ is a unit of  $K$, so the ideal  generated by
   $\alpha^2$  is equal to  $(\pi)$. Thus
   $(\alpha^2)=(\pi)=\mathcal{H}^2$, hence  $(\alpha)= \mathcal{H}$ and the result derived.

Similarly, the system (2) implies that  $2x=iy_1^2\pi_2-iy_2^2\pi_1$, thus $\alpha=\sqrt{\pi\epsilon_{p}}=\frac{1}{2}(y_1(1+i)\pi+y_2(1-i)\sqrt{p})\in K$. Hence $\pi\epsilon_{p}=\alpha^2$ and  $(\alpha)= \mathcal{H}$, so the result.
 \end{proof}
\begin{The}\label{239}
Keep the notations and hypotheses previously mentioned   and put  $\epsilon_{p_2q}=a+b\sqrt{p_2q}$,  $\epsilon_{p_1p_2q}=x+y\sqrt{p_1p_2q}$.
\begin{enumerate}[\rm\indent1.]
  \item If $x\pm1$ is a square in $\NN$ and $a+1$, $a-1$ are not, then $\kappa_{\KK_1}=\langle[\h], [\hh], [\hhh\hhhh]\rangle$.
  \item If $a\pm1$ and $(p_1(x\pm1)$ or $2p_1(x\pm1))$ are squares in $\NN$, then $\kappa_{\KK_1}=\langle[\h]\rangle$.
    \item If $a+1$, $a-1$ are not squares in  $\NN$ and $p_1(x\pm1)$ or $2p_1(x\pm1)$  is, then $\kappa_{\KK_1}=\langle[\h], [\hhh\hhhh]\rangle$.
  \item In the other cases we have: $\kappa_{\KK_1}=\langle[\h], [\hh]\rangle$.
\end{enumerate}
\end{The}
\begin{proof}By Proposition \ref{9}, $\h$ and $\hh$ capitulate in $\KK_1$.

1. As $a+1$ and $a-1$ are not squares in $\NN$, so, from the proof of   Proposition \ref{27},  we get  $p_2\epsilon_{p_2q}$ or $2p_2\epsilon_{p_2q}$ is a square in $\KK_1$. Therefore there exist $\alpha \in\KK_1$ such that $$(\alpha^2)=(p_2)=(\pi_2\pi_3)=(\hhh\hhhh)^2\text{ or } \left(\frac{\alpha}{1+i}\right)^2=(p_2)=(\pi_2\pi_3)=(\hhh\hhhh)^2.$$
Which implies that $\hhh\hhhh$ capitulates in $\KK_1$. On the other hand, since  $x\pm1$ is a square in $\NN$ and $a+1$, $a-1$ are not, so proceeding as in  Lemma \ref{7}, we prove that $\h\hh$, $\h\hhh\hhhh$, $\hh\hhh\hhhh$ and $\h\hh\hhh\hhhh$ are not principal in $\kk$. Hence Theorem \ref{226} implies the result.

2. Since $p_1(x\pm1)$ or $2p_1(x\pm1)$ are squares in $\NN$ and $(\h\hh)^2=(p_1)$, so, by \cite[Proposition 2]{AZT12-2}, $[\h]=[\hh]$. Hence Theorem \ref{226} implies the result.

 3. Since $[\h]=[\hh]$, the result is obvious by Theorem \ref{226}.

 4. As $p_1(x+1)$, $p_1(x-1)$,  $2p_1(x+1)$ and $2p_1(x-1)$ are not squares in $\NN$, so $[\h]\neq[\hh]$. Hence Theorem \ref{226} implies the result.
\end{proof}
\subsection{Capitulation in $\KK_2$}
Since $p_1$ and $p_2$ play symmetric roles, so the capitulation of the $2$-ideal classes of  $\kk$ in $\KK_2=\kk(\sqrt{p_2})$ is deduced from the  previous Theorem \ref{239}.
\begin{The}\label{241}
Keep the notations and hypotheses previously mentioned   and put  $\epsilon_{p_1q}=a+b\sqrt{p_1q}$.
\begin{enumerate}[\rm\indent1.]
  \item If $x\pm1$ is a square in  $\NN$ and  $a+1$, $a-1$ are not, then $\kappa_{\KK_2}=\langle[\hhh], [\hhhh], [\h\hh]\rangle$.
  \item If $a\pm1$ and $(p_2(x\pm1)$ or $2p_2(x\pm1))$ are squares in $\NN$, then $\kappa_{\KK_2}=\langle[\hhh]\rangle$.
    \item If $a+1$ and $a-1$ are not squares in $\NN$ and $p_2(x\pm1)$ or $2p_2(x\pm1)$  is, then $\kappa_{\KK_2}=\langle[\hhh], [\h\hh]\rangle$.
  \item In the other cases,  $\kappa_{\KK_2}=\langle[\hhh], [\hhhh]\rangle$.
\end{enumerate}
\end{The}
\subsection{Capitulation in $\KK_3$}
Finally, we study the capitulation of the $2$-ideal classes of  $\kk$ in  $\KK_3=\kk(\sqrt{q})=\kk(\sqrt{p_1p_2})$.
\begin{The}\label{243}
Keep the notations and hypotheses previously mentioned   and assume  $N(\epsilon_{p_1p_2})=1$.
\begin{enumerate}[\rm\indent1.]
  \item If $x\pm1$ is a square in $\NN$ , then $\kappa_{\KK_3}=\langle[\h\hh], [\hhh\hhhh]\rangle$.
  \item If  $ p_1(x\pm1)$ or $2p_1(x\pm1)$ is a square in $\NN$, then $\kappa_{\KK_3}=\langle[\hhh\hhhh]\rangle$.
    \item If  $p_2(x\pm1)$ or $2p_2(x\pm1)$  is a square in $\NN$, then $\kappa_{\KK_3}=\langle[\h\hh]\rangle$.
  \item If $q(x\pm1)$ or $2q(x\pm1)$  is a square in $\NN$, then $\kappa_{\KK_3}=\langle[\h\hh]\rangle=\langle[\hhh\hhhh]\rangle$.
\end{enumerate}
\end{The}
\begin{proof}
According to  the previous theorems   $\h$, $\hh$, $\hhh$ and $\hhhh$ are not principal in $\kk$. On the other hand,  from the proof of  Proposition \ref{29} we know that $2p_1\epsilon_{p_1p_2}$ and $2p_2\epsilon_{p_1p_2}$ or $p_1\epsilon_{p_1p_2}$ and  $p_2\epsilon_{p_1p_2}$ are squares in $\KK_3$. Thus there exist  $\alpha$, $\beta$ in $\KK_3$ such that $\h\hh=(\frac{\alpha}{1+i})$ and $\hhh\hhhh=(\frac{\beta}{1+i})$ or $\h\hh=(\alpha)$ and $\hhh\hhhh=(\beta)$, which implies that  $\h\hh$ and $\hhh\hhhh$ capitulate in $\KK_3$. We have four cases to distinguish.
\begin{enumerate}[\rm\indent1.]
 \item Suppose  $x\pm1$ is a square in $\NN$.  Since   $(\h\hh)^2=(p_1)$, $(\hhh\hhhh)^2=(p_2)$ and $(\h\hh\hhh\hhhh)^2=(p_1p_2)$, then  Proposition 2 and Remark  1 of  \cite{AZT12-2} yield that $\h\hh$, $\hhh\hhhh$ and $\h\hh\hhh\hhhh$ are not principal in  $\kk$; hence $\kappa_{\KK_3}=\langle[\h\hh], [\hhh\hhhh]\rangle$.
 \item  If $p_1(x\pm1)$ or $2p_1(x\pm1)$ is a square in  $\NN$, then \cite[Proposition 2]{AZT12-2} yields that  $[\h\hh]=1$; so the result.
 \item If  $p_2(x\pm1)$ or $2p_2(x\pm1)$  is a square in  $\NN$, then \cite[Proposition 2]{AZT12-2} yields that $[\hhh\hhhh]=1$; so the result.
 \item  If  $q(x\pm1)$ or $2q(x\pm1)$  is a square in  $\NN$, then \cite[Remark 1]{AZT12-2} yields that  $\h\hh\hhh\hhhh$ is principal in  $\kk$, so the result.
\end{enumerate}
\end{proof}
\begin{The}\label{245}
Keep the notations and hypotheses previously mentioned   and assume  $N(\epsilon_{p_1p_2})=-1$.
\begin{enumerate}[\rm\indent1.]
  \item If  $q(x\pm1)$ or $2q(x\pm1)$  is a square in $\NN$, then $\kappa_{\KK_3}=\langle[\h\hhh]\rangle$ or $\langle[\h\hhhh]\rangle$.
  \item If  $x\pm1$  is a square in $\NN$, then  $\kappa_{\KK_3}=\langle[\h\hhh], [\hh\hhhh]\rangle$ or $\langle[\h\hhhh], [\hh\hhh]\rangle$.
  \item If $p_1(x\pm1)$ or $2p_1(x\pm1)$ is a square in $\NN$, then  $\kappa_{\KK_3}=\langle[\h\hhh], [\h\hhhh]\rangle$.
  \item If $p_2(x\pm1)$ or $2p_2(x\pm1)$ is a square in $\NN$, then  $\kappa_{\KK_3}=\langle[\h\hhh], [\hh\hhh]\rangle$.
\end{enumerate}
\end{The}
\begin{proof}
As $N(\epsilon_{p_1p_2})=-1$, then by  the decomposition uniqueness in $\ZZ[i]$, there exist  $b_1$ and $b_2$ in $\ZZ[i]$ such that
 $$ \left\{\begin{array}{ll}
a\pm i &=b_1^2\pi_1\pi_3,\\
a\mp i &=b_2^2\pi_2\pi_4,
\end{array}\right. \text{ or }
  \left\{\begin{array}{ll}
a\pm i &=ib_1^2\pi_1\pi_3,\\
a\mp i &=-ib_2^2\pi_2\pi_4,
\end{array}\right. \text{ or }$$
 $$\left\{\begin{array}{ll}
a\pm i &=b_1^2\pi_1\pi_4,\\
a\mp i &=b_2^2\pi_2\pi_3,
\end{array}\right. \text{ or }
  \left\{\begin{array}{ll}
a\pm i &=ib_1^2\pi_1\pi_4,\\
a\mp i &=-ib_2^2\pi_2\pi3,
\end{array}\right.$$  with $p_1=\pi_1\pi_2$, $p_2=\pi_3\pi_4$, $y=y_1y_2$ and $\pi_2$ (resp. $\pi_4$, $y_2$) is the complex conjugate of $\pi_1$ (resp. $\pi_3$, $y_1$). Therefore $\sqrt{2\epsilon_{p_1p_2}}=b_1\sqrt{\pi_1\pi_3}+b_2\sqrt{\pi_2\pi_4}$ or $\sqrt{\epsilon_{p_1p_2}}=b_1(1+i)\sqrt{\pi_1\pi_3}+b_2(1-i)\sqrt{\pi_2\pi_4}$ or $\sqrt{2\epsilon_{p_1p_2}}=b_1\sqrt{\pi_1\pi_4}+b_2\sqrt{\pi_2\pi_3}$ or $\sqrt{\epsilon_{p_1p_2}}=b_1(1+i)\sqrt{\pi_1\pi_4}+b_2(1-i)\sqrt{\pi_2\pi_3}$, hence $2\pi_1\pi_3\epsilon_{p_1p_2}$ and  $2\pi_2\pi_4\epsilon_{p_1p_2}$ or  $\pi_1\pi_3\epsilon_{p_1p_2}$ and $\pi_2\pi_4\epsilon_{p_1p_2}$ or $2\pi_1\pi_4\epsilon_{p_1p_2}$ and  $2\pi_2\pi_3\epsilon_{p_1p_2}$ or  $\pi_1\pi_4\epsilon_{p_1p_2}$ and $\pi_2\pi_3\epsilon_{p_1p_2}$ are squares in $\KK_3$. Thus there exist    $\alpha$, $\beta$  in $\KK_3$  such that
 $(\alpha^2)= (2\pi_1\pi_3)$ and $(\beta^2)=(2\pi_2\pi_4)$ or  $(\alpha^2)= (\pi_1\pi_3)$ and $(\beta^2)=(\pi_2\pi_4)$ or  $(\alpha^2)= (2\pi_1\pi_4)$ and $(\beta^2)=(2\pi_2\pi_3)$ or $(\alpha^2)= (\pi_1\pi_4)$ and $(\beta^2)=(\pi_2\pi_3)$. Consequently $\h\hhh=(\frac{\alpha}{1+i})$ and $\hh\hhhh=(\frac{\beta}{1+i})$ or  $\h\hhh=(\alpha)$ and $\hh\hhhh=(\beta)$ or $\h\hhhh=(\frac{\alpha}{1+i})$ and $\hh\hhh=(\frac{\beta}{1+i})$ or  $\h\hhhh=(\alpha)$ and $\hh\hhh=(\beta)$. Thus $\h\hhh$ and $\hh\hhhh$ or $\h\hhhh$ and $\hh\hhh$ capitulate in $\KK_3$. On the other hand, $\h\hhh$ is not principal in  $\kk$, since $(\h\hhh)^2=(\pi_1\pi_3)=(m+in)$ and $\sqrt{m^2+n^2}=\sqrt{p_1p_2}\not\in\kk$. Thus \cite[Proposition 1]{AZT12-2} guarantees the result. We similarly show that $\h\hhhh$, $\hh\hhhh$ and $\hh\hhh$ are not principal in $\kk$.

1. If  $q(x\pm1)$ or $2q(x\pm1)$  is  a square in $\NN$, then $p_1p_2(x\pm1)$ or $2p_1p_2(x\pm1)$  is a square in $\NN$; so \cite[Remark 1]{AZT12-2} yields that $\h\hh\hhh\hhhh$ is principal in $\kk$, thus $\kappa_{\KK_3}=\langle[\h\hhh]\rangle$ or $\langle[\h\hhhh]\rangle$.

2., 3. and 4. are similarly shown.
\end{proof}

From Theorems \ref{239},   \ref{241}, \ref{243},  \ref{245} and from Proposition \ref{248}, we deduce the following result.
\begin{The}\label{10}
Let  $p_1\equiv p_2\equiv-q\equiv1\pmod4$ be different primes. Put $\kk=\QQ(\sqrt{p_1p_2q}, i)$ and denote by $\G$ its genus field. Let  $\epsilon_{p_1p_2q}=x+y\sqrt{p_1p_2q}$ be the fundamental unit of   $\QQ(\sqrt{p_1p_2q})$.
\begin{enumerate}[\rm\indent1.]
  \item If $x\pm1$ is a square in $\NN$, then $\langle\h, \hh, \hhh, \hhhh\rangle\subseteq \kappa_{\G}$.
  \item If $p_1(x\pm1)$ or $2p_1(x\pm1)$ is a square in $\NN$, then $\langle\h, \hhh, \hhhh\rangle\subseteq \kappa_{\G}$.
  \item If $p_2(x\pm1)$ or $2p_2(x\pm1)$ is a square in $\NN$, then $\langle\h, \hh, \hhh\rangle\subseteq \kappa_{\G}$.
  \item If $q(x\pm1)$ or $2q(x\pm1)$ is a square in $\NN$, then
   $\langle\h, \hh, \hhh\rangle=\langle\h, \hh, \hhhh\rangle\subseteq \kappa_{\G}$.
\end{enumerate}
\end{The}
 Theorem \ref{10} implies the following corollary:
\begin{coro}\label{19}
Let $\kk=\QQ(\sqrt{p_1p_2q},i)$, where $p_1\equiv p_2\equiv-q\equiv1 \pmod 4$ are different primes.   Let $\G$ be the genus field of $\kk$ and  $\mathrm{A}m_s(\kk/\QQ(i))$ be the group of  the strongly ambiguous class  of $\kk/\QQ(i)$, then $\mathrm{A}m_s(\kk/\QQ(i))\subseteq \kappa_{\k}$.
\end{coro}
\section{\bf{Application}}
Let $p_1\equiv p_2\equiv-q\equiv1 \pmod 4$ be different primes such that  $\mathbf{C}l_2(\mathds{\kk})$ is of type $(2, 2, 2)$.
According to  \cite{AzTa-08}, $\mathbf{C}l_2(\kk)$  is of type $(2, 2, 2)$ if and only if  $p_1$, $p_2$ and $q$  satisfy the following conditions:
 \begin{enumerate}[\rm\indent -]
\item  $p_1\equiv5$ or $p_2\equiv5\pmod8$.
\item  Two at least of the elements of  $\left\{\left(\frac{p_1}{p_2}\right), \left(\frac{p_1}{q}\right), \left(\frac{p_2}{q}\right)\right\}$ are equal to $-1$.
  \end{enumerate}
  These conditions are detailed in three types $I$, $II$ and $III$ and each type consists of three cases $(a)$, $(b)$ and $(c)$ (see \cite{AZT12-2}).
To continue we need the following results.
\begin{lem}\label{11}
Let $p_1\equiv p_2\equiv-q\equiv1 \pmod 4$ be different primes as above.
\begin{enumerate}[\upshape\indent1.]
  \item If  $p_1$, $p_2$ and $q$ are of  type I,  then  $p_1(x\pm1)$ is a square in $\NN$.
  \item If $p_1$, $p_2$ and $q$ are of  type II,  then  $p_2(x\pm1)$ is a square in $\NN$.
  \item If $p_1$, $p_2$ and $q$ are of  type III, then  $q(x\pm1)$ i.e. $p_1p_2(x\mp1)$ is a square in $\NN$.
\end{enumerate}
\end{lem}
\begin{proof}
As $p_1\equiv5$ or $p_2\equiv5\pmod8$, so the unit index of   $\kk$ is $1$ (see \cite[Corollaire 3.2]{AzTa-08}). On the other hand,
 $N(\varepsilon_d)=1$ i.e. $x^2-1=y^2p_1p_2q$, hence by  Lemma \ref{6},   $x\pm1$ is not a square in $\NN$. Thus by the decomposition uniqueness in $\ZZ$ and by Lemma \ref{5}, there exist $y_1$, $y_2$ in $\ZZ$ such that:\\
   (1) $\left\{\begin{array}{ll}
 x\pm1=p_1y_1^2,\\
 x\mp1=p_2qy_2^2;
 \end{array}\right.$ or (2)
    $\left\{\begin{array}{ll}
 x\pm1=2p_1y_1^2,\\
 x\mp1=2p_2qy_2^2;
 \end{array}\right.$ or (3)
    $\left\{\begin{array}{ll}
 x\pm1=p_2y_1^2,\\
 x\mp1=p_1qy_2^2;
 \end{array}\right.$ or\\ (4)
    $\left\{\begin{array}{ll}
 x\pm1=2p_2y_1^2,\\
 x\mp1=2p_1qy_2^2;
 \end{array}\right.$ or (5)
    $\left\{\begin{array}{ll}
 x\pm1=qy_1^2,\\
 x\mp1=p_1p_2y_2^2;
 \end{array}\right.$ or (6)
    $\left\{\begin{array}{ll}
 x\pm1=2qy_1^2,\\
 x\mp1=2p_1p_2y_2^2;
 \end{array}\right.$\par
1.  Suppose $p_1$, $p_2$ and $q$ are of type I, then this contradicts systems (2), (3), (4), (5) and (6), since:
\begin{enumerate}[\upshape\indent -]
 \item system (2) implies that $\left(\frac{p_1}{p_2}\right)= \left(\frac{p_1}{q}\right)=1$
  \item system (3) implies that $\left(\frac{p_1}{p_2}\right)=\left(\frac{2}{p_1}\right)$ and $\left(\frac{p_1q}{p_2}\right)=\left(\frac{2}{p_2}\right)$
   \item system (4) implies that $\left(\frac{p_1}{p_2}\right)=\left(\frac{p_2}{q}\right)=1$
   \item   system (5) implies that $\left(\frac{p_1}{q}\right)=\left(\frac{2}{p_1}\right)=1$
  \item system (6) implies that $\left(\frac{p_1}{q}\right)=1$.
   \end{enumerate}
Thus only the system (1) occurs, which yields that    $p_1(x\pm1)$ is a square in $\NN$ and $p_2(x\pm1)$, $2p_2(x\pm1)$ are not.

 2.  and 3. are similarly checked.
\end{proof}
  Proceeding similarly, we prove the following lemma.
\begin{lem}\label{12}
Let  $p_1\equiv p_2\equiv-q\equiv1\pmod4$ be different primes and put
$\varepsilon_{p_2q}=a+b\sqrt{p_2q}$.
\begin{enumerate}[\rm\indent1.]
  \item If $p_1$,  $p_2$ and $q$ are of type $I(a)$, then $p_2(a\pm1)$ or $2p_2(a\pm1)$ is a square in $\NN$.
  \item If $p_1$,  $p_2$ and $q$ are of type $I(b)$, then $p_2(a\pm1)$ is a square in $\NN$.
  \item If $p_1$,  $p_2$ and $q$ are of type $I(c)$ or $II(a)$, then $a\pm1$ is a square in $\NN$.
  \item If $p_1$,  $p_2$ and $q$ are of type $II(c)$ or $III(a)$ or $III(b)$, then $p_2(a\pm1)$ is a square in $\NN$.
 \item If $p_1$,  $p_2$ and $q$ are of type $II(b)$ or $III(c)$, then $2p_2(a\pm1)$ is a square in $\NN$.
\end{enumerate}
\end{lem}
Denote by $\h$ and $\hh$ (resp. $\hhh$ and $\hhhh$) the prime ideals of $\kk$ above $p_1$ (resp. $p_2$), then we have:
\begin{lem}[\cite{AZT12-2}]\label{13}
Let  $p_1\equiv p_2\equiv-q\equiv1\pmod4$  be different primes and assume $\mathbf{C}l_2(\mathds{\kk})\simeq (2, 2, 2)$.
\begin{enumerate}[\rm\indent1.]
  \item If $p_1$,  $p_2$ and $q$ are of type $I$, then $\mathbf{C}l_2(\mathds{\kk})=\langle [\mathcal{H}_1], [\mathcal{H}_3], [\mathcal{H}_4]\rangle$.
  \item If $p_1$,  $p_2$ and $q$ are of type $II$ or $III$, then $\mathbf{C}l_2(\mathds{\kk})=\langle [\mathcal{H}_1], [\mathcal{H}_2], [\mathcal{H}_3]\rangle$.
  \end{enumerate}
\end{lem}
\begin{rema}
If $\mathbf{C}l_2(\mathds{\kk})$ is of type $(2, 2, 2)$, then
by Proposition \ref{248} and Lemma \ref{13}, we deduce that $\mathrm{Am}(\kk/\QQ(i))=\mathrm{Am}_s(\kk/\QQ(i))=\mathbf{C}l_2(\mathds{\kk})$.
\end{rema}
\begin{The}\label{14}
Let $\kk=\K$,  where $d=p_1p_2q$ with $p_1$, $p_2$ and $q$ are different primes such that  $\mathbf{C}l_2(\kk)$, the   $2$-class groupe of $\kk$, is of type $(2, 2, 2)$.
\begin{enumerate}[\upshape\indent1.]
  \item If $p_1$,  $p_2$ and $q$ are of type $I(c)$, then $\kappa_{\KK_1}=\langle[\h]\rangle$.
  \item If $p_1$,  $p_2$ and $q$ are of type $I(a)$ or $I(b)$, then $\kappa_{\KK_1}=\langle[\h], [\hhh\hhhh]\rangle$.
  \item If $p_1$,  $p_2$ and $q$ are of type $II$ or $III$, then $\kappa_{\KK_1}=\langle[\h], [\hh]\rangle$.
\end{enumerate}
\end{The}
\begin{proof}
From  Lemmas \ref{11},  \ref{12} and \ref{13} we get:
\begin{enumerate}[\upshape\indent1.]
\item  If $p_1$,  $p_2$ and $q$ are of type $I(c)$, then  $a\pm1$ and $p_1(x\pm1)$ are squares in  $\NN$. Hence Theorem \ref{239} implies the result.
\item If $p_1$,  $p_2$ and $q$ are of type $I(a)$ or $I(b)$, then $p_2(a\pm1)$ or $2p_2(a\pm1)$ is a square in $\NN$,  and since $p_1(x\pm1)$ is a square in   $\NN$, hence Theorem \ref{239} implies the result.
\item a. If $p_1$,  $p_2$ and $q$ are of type $II$, then $a\pm1$ or $p_2(a\pm1)$ or $2p_2(a\pm1)$  is a square in $\NN$, and as in this case  $p_2(x\pm1)$ is also a square in $\NN$, hence Theorem \ref{239} implies the result.\\
 b. If $p_1$,  $p_2$ and $q$ are of type $III$, then $p_2(a\pm1)$ or $2p_2(a\pm1)$  is a square in $\NN$, and since  $q(x\pm1)$ is also a square in dans $\NN$, hence Theorem \ref{239} implies the result.
\end{enumerate}
\end{proof}
As $p_1$ and $p_2$ play symmetric roles, so with a similar argument to that used in the previous theorem, we deduce the following theorem. Note that in this case $\hhh$ and $\hhhh$ always capitulate in $\KK_2$ (Proposition \ref{9}). Note also that whenever   $p_1$, $p_2$ and $q$ are of type $II$, then $[\hhh]=[\hhhh]$ since in this case $p_2(x\pm1)$ is a square in $\NN$ and the result is guaranteed by \cite[Proposition 1]{AZT12-2}. Finally,  note that if $p_1$, $p_2$ and $q$ are of type $III$, then $\mathcal{Q}$, the prime ideal of $\kk$ lies above $q$, is principal in $\kk$; hence $\h\hh\hhh\hhhh$ is too.
\begin{The}\label{15}
Keep the hypotheses and notations mentioned in  Theorem $\ref{14}$.
\begin{enumerate}[\upshape\indent1.]
  \item If $p_1$,  $p_2$ and $q$ are of type $II(c)$, then $\kappa_{\KK_2}=\langle[\hhh]\rangle$.
  \item If $p_1$,  $p_2$ and $q$ are of type $II(a)$ or $II(b)$ or $III$, then $\kappa_{\KK_2}=\langle[\hhh], [\h\hh]\rangle$.
  \item If $p_1$,  $p_2$ and $q$ are of type $I$, then $\kappa_{\KK_2}=\langle[\hhh], [\hhhh]\rangle$.
\end{enumerate}
\end{The}
Finally, we compute the 2-idea classes of $\kk$ that capitulate in $\KK_3=\QQ(\sqrt{q}, \sqrt{p_1p_2}, i)$.
\begin{The}\label{16}
Keep the hypotheses and notations mentioned in  Theorem $\ref{14}$ and assume   $N(\varepsilon_{p_1p_2})=1$.
\begin{enumerate}[\upshape\indent1.]
  \item If $p_1$,  $p_2$ and $q$ are of type $I$, then $\kappa_{\KK_3}=\langle[\hhh\hhhh]\rangle$.
  \item If $p_1$,  $p_2$ and $q$ are of type $II$  or $III$, then $\kappa_{\KK_3}=\langle[\h\hh]\rangle$.
\end{enumerate}
\end{The}
\begin{proof}
From  Lemmas \ref{11} and \ref{13} we get:
\begin{enumerate}[\upshape\indent1.]
\item  If $p_1$,  $p_2$ and $q$ are of type $I(c)$, then  then  $p_1(x\pm1)$ is a square in  $\NN$. Hence Theorem \ref{243} implies the result.
\item a. If $p_1$,  $p_2$ and $q$ are of type $II$, then  $p_2(x\pm1)$   is a square in $\NN$, hence Theorem \ref{243} implies the result.\\
 b. If $p_1$,  $p_2$ and $q$ are of type $III$, then $q(x\pm1)$  i.e. $p_1p_2(x\pm1)$    is a square in $\NN$, hence Theorem \ref{243} implies the result.
\end{enumerate}
\end{proof}
\begin{The}\label{17}
Keep the hypotheses and notations mentioned in  Theorem $\ref{14}$ and assume   $N(\varepsilon_{p_1p_2})=-1$.
\begin{enumerate}[\upshape\indent1.]
  \item If $p_1$,  $p_2$ and $q$ are of type $III$, then $\kappa_{\KK_3}=\langle[\h\hhh]\rangle$ or $\langle[\hh\hhh]\rangle$.
  \item If $p_1$,  $p_2$ and $q$ are of type  $II$, then $\kappa_{\KK_3}=\langle[\h\hhh], [\hh\hhh]\rangle$.
  \item If $p_1$,  $p_2$ and $q$ are of type $I$, then $\kappa_{\KK_3}=\langle[\h\hhh], [\h\hhhh]\rangle$.
\end{enumerate}
\end{The}
\begin{proof}It is  a simple deduction from Theorem \ref{245} and Lemma \ref{11}.
\end{proof}
From Theorems \ref{14}, \ref{15}, \ref{16} and \ref{17}, we deduce the following corollary .
\begin{coro}\label{18}
Keep the hypotheses and notations mentioned in  Theorem $\ref{14}$. Then all the classes of   $\mathbf{C}l_2(\kk)$ capitulate in $\G$ i.e.
$$\kappa_{\G}=\mathbf{C}l_2(\kk)=\mathrm{Am}(\kk/\QQ(i))=\mathrm{Am}_s(\kk/\QQ(i)).$$
\end{coro}
\section{Proof of the main Theorem}
The main Theorem is a simple deduction from Proposition \ref{248}, Theorems \ref{226}, \ref{239}, \ref{241}, \ref{243}, \ref{245} and Corollaries \ref{19}, \ref{18}.

\end{document}